\documentclass[11pt]{article}
\usepackage[margin=1in]{geometry}
\usepackage{amsmath,amssymb,amsthm}
\usepackage[round,authoryear]{natbib}
\usepackage{microtype}
\usepackage{xcolor}
\definecolor{zaffre}{HTML}{0014A8}
\usepackage[colorlinks=true, linkcolor=zaffre, citecolor=zaffre, urlcolor=zaffre]{hyperref}

\newcommand{\Ee}{\mathbb{E}}
\newcommand{\Pp}{\mathbb{P}}
\newcommand{\RR}{\mathbb{R}}
\newcommand{\one}{\mathbf{1}}

\DeclareMathOperator{\Dir}{Dir}

\theoremstyle{plain}
\newtheorem{theorem}{Theorem}
\newtheorem{proposition}[theorem]{Proposition}
\newtheorem{lemma}[theorem]{Lemma}
\theoremstyle{remark}
\newtheorem{remark}[theorem]{Remark}

\title{An Exact Distribution-Free Test for \\ Means of Nonnegative Random Variables}

\author{Nikos Vlassis\\Adobe Research \and Philip S. Thomas\\University of Massachusetts}

\date{}

\begin{document}
\maketitle

\begin{abstract}
Let $X=(X_1,\ldots,X_n)$ be independent nonnegative random variables, not necessarily identically distributed. Let $D=(D_0,D_1,\ldots,D_n)\sim\Dir(1,\ldots,1)$ be independent of $X$, and define $K(x)=\Pp\{\sum_{i=1}^n x_iD_i\le1\}$. We prove that, for every $n\ge1$, whenever $\Ee X_i\le1$ for every~$i$, $\Pp\{K(X)\le\alpha\}\le\alpha$ for all $0\le\alpha\le1$. Thus $K(X)$ is a finite-sample, distribution-free $p$-value for testing the null hypothesis $\Ee X_i \le 1$ for all $i$. This proves a conjecture of \citet{Gaffke}.
\end{abstract}

\section{Introduction}
\label{sec:intro}

Let $X=(X_1,\ldots,X_n)$ be a vector of independent nonnegative random variables. We consider the one-sided testing problem $H_0:\Ee X_i \le1$ for all $i$ against the alternative that $\Ee X_i>1$ for some $i$. The threshold $1$ entails no loss of generality, since a positive common threshold can be reduced to this case by rescaling the variables. The problem has little distributional structure: the variables need not be identically distributed, continuous, or subject to any shape constraint, yet the goal is a test with finite-sample validity. \citet{Gaffke} proposed a test statistic that uses the available structure efficiently: augment the sample with a zero, average the resulting $n+1$ values using uniform Dirichlet weights, and take the conditional probability that this average is at most $1$. We prove that this probability is a valid one-sided $p$-value under the stated model assumptions and $H_0$.

For $x\in[0,\infty)^n$, let $K(x)=\Pp\{\sum_{i=1}^n x_iD_i\le1\}$, where $D=(D_0,D_1,\ldots,D_n)\sim\Dir(1,\ldots,1)$ is a vector of Dirichlet weights. Equivalently, using the representation $D_i=E_i/\sum_{r=0}^n E_r$, where $E_0,\ldots,E_n$ are i.i.d.\ exponential random variables with rate $1$, we have
\begin{equation}
    \label{eq:Kexp}
    K(x)=\Pp\Big\{\textstyle\sum_{i=1}^n (x_i-1)E_i\le E_0\Big\}.
\end{equation}
It follows that $K$ is nonincreasing in each coordinate and that $K(x)=1$ whenever $x_i\le 1$ for all $i$. The corresponding level-$\alpha$ test rejects when $K(X) \le\alpha$.

\begin{theorem}
    \label{thm:main}
    If $X_1,\ldots,X_n$ are independent, nonnegative, and satisfy $\Ee X_i\le1$, then $\Pp\{K(X)\le\alpha\}\le\alpha$ for every $0\le\alpha\le1$.
\end{theorem}

\citet{Gaffke} introduced $K$ and conjectured this finite-sample validity. For i.i.d.\ samples with common mean $\mu$, he established the limiting behavior of $K$, which converges almost surely to $1$ for $\mu<1$, almost surely to $0$ for $\mu>1$, and in distribution to $\mathrm{Unif}(0,1)$ for $\mu=1$ with finite, positive variance. He further reduced the non-identically distributed problem, at each fixed level~$\alpha$, to independent mean-one two-point marginals, proved the case $n=2$ within that family, and reported numerical verification up to $n=15$. In the related confidence-bound formulation for i.i.d.\ variables supported on $[0,1]$, \citet{LearnedMillerThomas} proved guaranteed coverage for Bernoulli and half-Bernoulli distributions. We prove the conjecture in full.\footnote{AI tools assisted with the development of this proof, including ideation, derivations, and writing.}

\paragraph{Proof outline.}

We prove the theorem first for mean-one two-point variables (Section~\ref{sec:twopoint}, with the key lemma proved in Section~\ref{sec:locallift}); the general case then follows by decomposition and rescaling (Section~\ref{sec:global}). In the two-point system, each $X_i$ takes a low value in $[0,1]$ or a high value above $1$, with the probabilities forced by the mean-one constraint. Recording a random outcome by the set $A\subseteq[n]=\{1,\ldots,n\}$ of variables that came out high, the law of $A$ is a product measure $\pi$, the statistic depends only on the high set, with value $K(A)$, and the statement to prove reads $\pi\{S:K(S)\le\alpha\}\le\alpha$.

The key result is that, for every payoff function $h:2^{[n]}\to\RR$ that is increasing with respect to set inclusion, there is a maximal chain ${\mathcal C}$ in $2^{[n]}$, and a measure $\nu_{\mathcal C}$ calibrated to $K$ along the chain (Section~\ref{sec:twopoint}), such that $\Ee_{\pi}h\le\Ee_{\nu_{\mathcal C}}h$. Taking $h$ to be the indicator of the rejection set $\{S:K(S)\le\alpha\}$, increasing by monotonicity of $K$, calibration bounds its $\nu_{\mathcal C}$-mass by $\alpha$, which proves the two-point case (Proposition~\ref{prop:twopoint}).

The dominating chain is built by induction. With the $X_i$ indexed so that their low values are in nondecreasing order, stage $k$ moves $X_k$ into the chain of $X_1,\ldots,X_{k-1}$, at a position chosen by a local lemma (Section~\ref{sec:locallift}), so that the mean of $h$ under the hybrid measure of placed and unplaced variables never decreases. The local lemma exhibits this insertion as an upward mass transport, controlled by a Stein-type identity for exponential shifts and a likelihood-ratio inequality.

\section{Two-point systems, chain measures, and validity}
\label{sec:twopoint}

We first prove Theorem~\ref{thm:main} for mean-one two-point variables $X_i$; the reduction to general marginals is given in Section~\ref{sec:global}. The two-point system is parametrized by
\begin{equation}
    \label{eq:twopointlaw}
    X_i\in\{1-\gamma_i,\ 1+\beta_i\},
    \qquad 0\le\gamma_i\le1,
    \qquad \beta_i>0,
\end{equation}
with
\begin{equation}
    \label{eq:pi}
    \Pp\{X_i=1+\beta_i\}=p_i=\frac{\gamma_i}{\gamma_i+\beta_i},
    \qquad
    \Pp\{X_i=1-\gamma_i\}=1-p_i=\frac{\beta_i}{\gamma_i+\beta_i}.
\end{equation}
Then $\Ee X_i=1$. We index the variables as $X_1,X_2,\ldots,X_n$ so that $\gamma_1\ge\gamma_2\ge\cdots\ge\gamma_n$. Since $K$ is invariant under simultaneous relabeling of the variables and their parameters, this sorting entails no loss of generality.

An outcome of the system is encoded by its high set. For $0\le m\le n$ and $S\subseteq[m]$, let $K_m(S)$ denote the value of $K$ at the outcome of the subsystem $X_1,\ldots,X_m$ with high set $S$:
\begin{equation}
    \label{eq:Kcube}
    K_m(S)=
    \Pp\left\{
        \sum_{i\in S}\beta_iE_i
        \le
        E_0+\sum_{i\in[m]\setminus S}\gamma_iE_i
    \right\},
    \qquad S\subseteq[m],
\end{equation}
where $E_0,E_1,\ldots,E_n$ are independent unit exponential variables; write $K=K_n$. For $0\le k\le n$, let $\pi_{>k}$ be the law of the high set of the tail variables $X_{k+1},\ldots,X_n$, the product measure on $2^{\{k+1,\ldots,n\}}$:
\begin{equation}
    \label{eq:productmeasure}
    \pi_{>k}(\{T\})
    =\prod_{i\in T}p_i\prod_{k<i\le n,\ i\notin T}(1-p_i),
    \qquad T\subseteq\{k+1,\ldots,n\}.
\end{equation}
Write $\pi=\pi_{>0}$ for the level-$n$ product law; if $A$ is the random set of high variables, then $A\sim\pi$ and $K(X)=K(A)$. For $m\ge1$, $K_m(\varnothing)=1$ follows from \eqref{eq:Kcube}, the left sum being empty; at $m=0$ we use the convention $K_0(\varnothing)=1$. (The tail law needs no convention: at $k=n$ the empty products in \eqref{eq:productmeasure} give $\pi_{>n}(\{\varnothing\})=1$.)

\begin{lemma}[Monotonicity]
    \label{lem:mono}
    For fixed sets $A\subseteq B\subseteq[m]$, we have $K_m(B)\le K_m(A)$.
\end{lemma}
\begin{proof}
    It suffices to consider $B=A\cup\{i\}$. The defining events of $K_m(A)$ and $K_m(B)$ differ only in the contribution of variable $i$: $\gamma_iE_i$ on the right side for $A$, $\beta_iE_i$ on the left side for $B$. Since both coefficients are nonnegative, the event for $B$ is contained in the event for $A$.
\end{proof}

We now define the objects the proof works with. A \emph{payoff} is a function $h:2^{[n]}\to\RR$; it is \emph{increasing} if $h(A)\le h(B)$ for all fixed sets $A\subseteq B\subseteq[n]$. A \emph{maximal chain at level $m$}, for $0\le m\le n$, is an increasing sequence $\varnothing=S_0\subset S_1\subset\cdots\subset S_m=[m]$ in which consecutive sets differ by exactly one variable. (All chains below are maximal, and we drop the qualifier.) By Lemma~\ref{lem:mono}, the numbers
\begin{equation}
    \label{eq:chainmeasure}
    q_j=K_m(S_j)-K_m(S_{j+1}),
    \quad 0\le j\le m-1,
    \qquad
    q_m=K_m(S_m),
\end{equation}
are nonnegative, and they telescope to $\sum_{j=0}^{m} q_j=K_m(S_0)=K_m(\varnothing)=1$. The \emph{chain measure} $\nu_{\mathcal C}$ places mass $q_j$ on $S_j$; we call the sets $S_j$ the \emph{states} of the chain.

Chains correspond to orderings: every ordering (permutation) $\sigma$ of the variables $X_1,\ldots,X_m$ defines a chain at level $m$ via $S_j=\{\sigma(1),\ldots,\sigma(j)\}$, for $j=1,\ldots,m$ (and $S_0=\varnothing$ always), and every chain arises from a unique ordering. The dominating chain is built by induction, one stage at a time, starting from the trivial chain $\mathcal{C}_0=(\varnothing)$ at level $0$: stage $k$ inserts $X_k$ into the ordering of the chain $\mathcal{C}_{k-1}$ at level $k-1$ produced by the previous stages, yielding a chain $\mathcal{C}_k$ at level $k$. (The insertion may change the states after the insertion point, so $\mathcal{C}_k$ need not be an extension of $\mathcal{C}_{k-1}$; only the relative order of $X_1,\ldots,X_{k-1}$ is preserved.) Each stage of the induction defines a measure on the fixed cube $2^{[n]}$: after stage $k$, the placed variables $X_1,\ldots,X_k$ follow the chain measure of $\mathcal{C}_k$, while the tail variables $X_{k+1},\ldots,X_n$ remain independent. Write $\mu_k$ for this \emph{hybrid} measure,
\begin{equation}
    \label{eq:hybrid}
    \mu_k=\nu_{\mathcal{C}_k}\otimes\pi_{>k},
\end{equation}
the law of $S\cup T_k$, where $S\sim\nu_{\mathcal{C}_k}$ and $T_k\sim\pi_{>k}$ is an independent high set of the tail. Under $\mu_k$ the placed variables are maximally dependent and the rest untouched. At the two ends, $\mu_0=\pi$ (the level-$0$ chain is trivial, so all variables are independent) and $\mu_n=\nu_{\mathcal{C}_n}$ (no tail remains).

The insertion position at stage $k$ is chosen according to Lemma~\ref{lem:cutlift}, so that the mean of an increasing payoff never decreases: $\Ee_{\mu_{k-1}}h\le\Ee_{\mu_k}h$, for $k=1,\ldots,n$. Granting that lemma, the two-point case of Theorem~\ref{thm:main} follows:

\begin{proposition}[Two-point validity]
    \label{prop:twopoint}
    For every independent mean-one two-point system \eqref{eq:twopointlaw}--\eqref{eq:pi} and every $\alpha\in[0,1]$,
    \begin{equation}
        \label{eq:twopointValidity}
        \pi\{S:K(S)\le\alpha\}\le\alpha.
    \end{equation}
\end{proposition}
\begin{proof}
    Variables with $\gamma_i=0$ are deterministic at $1$ and contribute $(X_i-1)E_i=0$ in \eqref{eq:Kexp}; deleting them changes neither the realized value of $K(X)$ nor the probability in question, so we may assume $\gamma_i>0$ for all $i$, and the remaining variables are still sorted. (If every variable is deleted, the system that remains is the $n=0$ convention.)

    Fix $\alpha\in[0,1]$. The payoff $h=\one_{\{K(\cdot)\le\alpha\}}$ is increasing by Lemma~\ref{lem:mono}, so the chains $\mathcal C_1,\ldots,\mathcal C_n$ produced by the induction for this $h$ satisfy
    \begin{equation}
        \Ee_{\pi}h
        =\Ee_{\mu_0}h
        \le\Ee_{\mu_1}h
        \le\cdots
        \le\Ee_{\mu_n}h
        =\Ee_{\nu_{\mathcal{C}_n}}h,
    \end{equation}
    each inequality being an application of Lemma~\ref{lem:cutlift} at the corresponding stage. Write the final chain $\mathcal C_n: \varnothing=S_0\subset\cdots\subset S_n=[n]$. Since the rejection set $\{S:K(S)\le\alpha\}$ is upward closed and the chain is ascending, the rejected states form a terminal segment $S_t\subset\cdots\subset S_n$ of the chain, possibly empty; its mass under $\nu_{\mathcal C_n}$ telescopes, by \eqref{eq:chainmeasure}, to $\sum_{j=t}^{n}q_j=K(S_t)$. Hence $\Ee_{\nu_{\mathcal C_n}}h=K(S_t)\le\alpha$ if some state is rejected, and $\Ee_{\nu_{\mathcal C_n}}h=0\le\alpha$ otherwise. This proves~\eqref{eq:twopointValidity}.
\end{proof}

\begin{remark}
    \label{rem:mixture}
    The dominating chain constructed by the induction depends on the payoff $h$, but this dependence can be removed: there is a random ordering $\sigma$ of $X_1,\ldots,X_n$, with law depending only on the parameters $\gamma_i,\beta_i$ of the system, whose chain measure $\nu_\sigma$ dominates~$\pi$ on average, i.e., $\Ee_\pi h\le\Ee_\sigma[\Ee_{\nu_\sigma}h]$ for every increasing $h$. This single measure certifies \eqref{eq:twopointValidity} at every level $\alpha$ at once. The law of $\sigma$ can be obtained by inserting each variable at a position drawn at random with the weights $\lambda_J$ of Lemma~\ref{lem:cutlift}; the domination follows by iterating \eqref{eq:cutliftReward}.
\end{remark}

\section{Proof of the local insertion lemma}
\label{sec:locallift}

This section proves the lemma that selects the insertion position at each stage of the induction. Assume throughout, as in the proof of Proposition~\ref{prop:twopoint}, that the variables with $\gamma_i=0$ have been deleted, so that $0<\gamma_i\le1$ and $\beta_i>0$ for all $i$, with $\gamma_1\ge\cdots\ge\gamma_n$. Fix a stage $k\in\{1,\ldots,n\}$, and let $\mathcal C=\mathcal C_{k-1}$ be the chain produced by the previous stages, with states $\varnothing=C_0\subset C_1\subset\cdots\subset C_{k-1}=[k-1]$. Set $H_j=C_j\cup\{k\}$, $c=\gamma_k$, $d=\beta_k$, and $p=p_k=c/(c+d)$. Inserting $X_k$ after $C_J$ gives the chain at level $k$
\begin{equation}
    \label{eq:liftedpath}
    C_0\subset\cdots\subset C_J\subset H_J\subset H_{J+1}\subset\cdots\subset H_{k-1},
\end{equation}
denoted $\mathcal C^J$, for $0\le J\le k-1$, with chain measure $\nu_{\mathcal C^J}$ computed with $K_k$. Write also $\widehat\nu$ for $\nu_{\mathcal C}$ with $X_k$ revealed independently: the law of $S\cup T$, where $S\sim\nu_{\mathcal C}$ and $T=\{k\}$ with probability $p$, else $T=\varnothing$. Section~\ref{sec:twopoint} uses Lemma~\ref{lem:cutlift} below with $g(S)=\Ee\,h(S\cup T_k)$, $T_k\sim\pi_{>k}$: then $\Ee_{\nu_{\mathcal C^J}}g$ is the hybrid mean \eqref{eq:hybrid} of the inserted chain, and $\Ee_{\widehat\nu}\,g=\Ee_{\mu_{k-1}}h$, the stage-$(k-1)$ tail being an independent reveal of $X_k$ together with $T_k$.

For $0\le j\le k-1$, let $G_j$ be the difference of the two sides of the comparison in \eqref{eq:Kcube} defining $K_{k-1}(C_j)$,
\begin{equation}
    \label{eq:slack}
    G_j=E_0+
      \sum_{\ell\in[k-1]\setminus C_j}\gamma_\ell E_\ell
      -\sum_{\ell\in C_j}\beta_\ell E_\ell.
\end{equation}
Along the chain, consecutive $G_{j-1}$ and $G_j$ agree except for one exponential, which enters $G_{j-1}$ with coefficient $+\gamma_\ell$ and $G_j$ with $-\beta_\ell$. The following lemma, the analytic core of the proof, compares such a pair.

\begin{lemma}[Exponential transfer]
    \label{lem:transfer}
    Let $Y$, $E$, $E'$ be independent random variables, with $E$ and $E'$ unit exponential, and let $a,b,c,d>0$. Set
    \begin{equation}
        \label{eq:Zpm}
        Z_+=Y+aE,
        \qquad
        Z_-=Y-bE,
    \end{equation}
     and, for $\varepsilon\in\{+,-\}$,
    \begin{equation}
        F_\varepsilon=\Pp\{Z_\varepsilon\ge0\},
        \qquad
        A_\varepsilon=\Pp\{Z_\varepsilon\ge dE'\},
        \qquad
        B_\varepsilon=\Pp\{Z_\varepsilon\ge -cE'\}.
    \end{equation}
    Write $u_\varepsilon=F_\varepsilon-A_\varepsilon$, $w_\varepsilon=B_\varepsilon-A_\varepsilon$, and $\theta_\varepsilon=u_\varepsilon/w_\varepsilon$. Then $w_+,w_->0$ and, with $p=c/(c+d)$,
    \begin{equation}
        \label{eq:transfer}
        (1-\theta_-)(A_+-A_-)-p\,(F_+-F_-)
        =\frac{a-c}{c+d}\,w_+\,(\theta_+-\theta_-).
    \end{equation}
    Moreover, if $Y$ has a log-concave density, then
    \begin{equation}
        \label{eq:monge}
        \theta_+\ \ge\ \theta_-.
    \end{equation}
\end{lemma}

\begin{proof}
    Since $E'$ is independent of $Z_\varepsilon$ and $\Pp\{E'>t\}=e^{-t}$ for $t\ge0$, conditioning on $Z_\varepsilon$ gives $A_\varepsilon=\Ee\,\alpha(Z_\varepsilon)$ and $B_\varepsilon=\Ee\,\beta(Z_\varepsilon)$, with
    \begin{equation}
        \label{eq:alphabeta}
        \alpha(x)=(1-e^{-x/d})\,\one_{\{x\ge0\}},
        \qquad
        \beta(x)=\one_{\{x\ge0\}}+e^{x/c}\,\one_{\{x<0\}},
    \end{equation}
    and hence, writing $v_\varepsilon=B_\varepsilon-F_\varepsilon$,    
    \begin{equation}
        \label{eq:windows}
        u_\varepsilon=\Ee\bigl[\one_{\{Z_\varepsilon\ge0\}}e^{-Z_\varepsilon/d}\bigr],
        \qquad
        v_\varepsilon=\Ee\bigl[\one_{\{Z_\varepsilon<0\}}e^{Z_\varepsilon/c}\bigr].
    \end{equation}
    In particular $w_\varepsilon=u_\varepsilon+v_\varepsilon=\Ee\bigl[e^{-Z_\varepsilon/d}\one_{\{Z_\varepsilon\ge0\}}+e^{Z_\varepsilon/c}\one_{\{Z_\varepsilon<0\}}\bigr]>0$, the integrand being strictly positive everywhere, so $\theta_\varepsilon$ is well defined.
    
    The identity \eqref{eq:transfer} follows from a Stein-type identity for the two-sided exponential shift: for every bounded Lipschitz $\phi$,
    \begin{equation}
        \label{eq:stein}
        \Ee\,\phi(Z_+)-\Ee\,\phi(Z_-)
        =a\,\Ee\,\phi'(Z_+)+b\,\Ee\,\phi'(Z_-),
    \end{equation}
    because $\phi(Z_+)-\phi(Z_-)=\int_{-bE}^{aE}\phi'(Y+t)\,dt$, whose expectation, computed by conditioning on $Y$, equals $a\,\Ee\,\phi'(Y+aE)+b\,\Ee\,\phi'(Y-bE)$ by Fubini, using $\Pp\{aE>t\}=e^{-t/a}$. The functions \eqref{eq:alphabeta} are bounded and Lipschitz, and differentiable except at $x=0$, with $\alpha'(x)=\tfrac1d\,e^{-x/d}\,\one_{\{x>0\}}$ and $\beta'(x)=\tfrac1c\,e^{x/c}\,\one_{\{x<0\}}$. The variables $Z_+$ and $Z_-$ have densities, since their independent summands $aE$ and $-bE$ do, so the point $x=0$ carries no mass, and \eqref{eq:windows} gives $\Ee\,\alpha'(Z_\varepsilon)=u_\varepsilon/d$ and $\Ee\,\beta'(Z_\varepsilon)=v_\varepsilon/c$. Taking $\phi=\alpha$ and $\phi=\beta$ in \eqref{eq:stein},
    \begin{equation}
        \label{eq:flux}
        d\,(A_+-A_-)=a\,u_++b\,u_-,
        \qquad
        c\,(B_+-B_-)=a\,v_++b\,v_-.
    \end{equation}
    Adding the two identities and substituting $B_\varepsilon=A_\varepsilon+w_\varepsilon$ gives the first identity below; rewriting the second through $F_\varepsilon=B_\varepsilon-v_\varepsilon$ gives the second:
    \begin{equation}
        \label{eq:combined}
        (c+d)(A_+-A_-)=(a-c)\,w_++(b+c)\,w_-,
        \qquad
        c\,(F_+-F_-)=(a-c)\,v_++(b+c)\,v_-.
    \end{equation}
    Now multiply the left side of \eqref{eq:transfer} by $c+d$; since $p\,(c+d)=c$, substituting \eqref{eq:combined} turns it into $(1-\theta_-)\bigl[(a-c)\,w_++(b+c)\,w_-\bigr]-\bigl[(a-c)\,v_++(b+c)\,v_-\bigr]$. Since $v_\varepsilon=(1-\theta_\varepsilon)\,w_\varepsilon$, the $(b+c)$ terms cancel and the $(a-c)$ terms leave $(a-c)\,w_+(\theta_+-\theta_-)$; dividing by $c+d$ proves \eqref{eq:transfer}.

    Now suppose $Y$ has a log-concave density $f$. Conditioning on $E$, the variables $Z_+$ and $Z_-$ have the densities
    \begin{equation}
        \label{eq:densities}
        f_+(x)=\int_0^\infty f(x-as)\,e^{-s}\,ds,
        \qquad
        f_-(x)=\int_0^\infty f(x+bs)\,e^{-s}\,ds.
    \end{equation}
    Fix $x\ge y$ and $s,t\ge0$: the numbers $x-as$ and $y+bt$ have the same sum as $x+bt$ and $y-as$ and lie between them, so they are complementary convex combinations of the latter, and adding the two concavity inequalities for $\log f$, valued in $[-\infty,\infty)$, gives $f(x-as)\,f(y+bt)\ge f(x+bt)\,f(y-as)$, the right side being $0$ when $f$ vanishes at $x+bt$ or $y-as$. Multiplying by $e^{-s-t}$ and integrating over $s,t\ge0$ yields the likelihood-ratio inequality
    \begin{equation}
        \label{eq:lr}
        f_+(x)\,f_-(y)\ \ge\ f_-(x)\,f_+(y),
        \qquad x\ge y.
    \end{equation}
    Finally, multiply \eqref{eq:lr} by $e^{-x/d}\,e^{y/c}>0$ and integrate over $x\ge0>y$: by \eqref{eq:windows}, which read $u_\varepsilon=\int_0^\infty e^{-x/d}f_\varepsilon(x)\,dx$ and $v_\varepsilon=\int_{-\infty}^0 e^{x/c}f_\varepsilon(x)\,dx$ in density form, this gives $u_+v_-\ge u_-v_+$, that is, $w_+w_-\,(\theta_+-\theta_-)\ge0$, and hence \eqref{eq:monge}.
\end{proof}

The inequality \eqref{eq:lr} is an instance of the preservation of total positivity under composition with the exponential translation kernel; see \citet[Chap.~1]{Karlin}.

To apply the lemma along the chain, it will be convenient to append a sentinel (fake) $G_k$, in which the term of $E_0$, too, is flipped:
\begin{equation}
    \label{eq:slackTop}
    G_k=-E_0-\sum_{\ell\in[k-1]}\beta_\ell E_\ell.
\end{equation}
For $0\le j\le k$, define, as in Lemma~\ref{lem:transfer} with $E'=E_k$,
\begin{equation}
    \label{eq:AFB}
    F_j=\Pp\{G_j\ge0\},
    \qquad
    A_j=\Pp\{G_j\ge dE_k\},
    \qquad
    B_j=\Pp\{G_j\ge -cE_k\},
\end{equation}
and the derived $u_j=F_j-A_j$, $v_j=B_j-F_j$, $w_j=B_j-A_j$, and $\theta_j=u_j/w_j$. For $j\le k-1$ these are three values of the statistic at the state $C_j$, with the new variable $X_k$ absent, high, and low: $F_j=K_{k-1}(C_j)$, and, since variable $k$ contributes $dE_k$ to the left side of the comparison when high and $cE_k$ to the right side when low, $A_j=K_k(H_j)$ and $B_j=K_k(C_j)$. The three events in \eqref{eq:AFB} are nested, so $u_j,v_j\ge0$; and conditioning on $G_j$, which is independent of $E_k$, exhibits $w_j=\Pp\{-cE_k\le G_j<dE_k\}$ as the expectation of a strictly positive function of $G_j$, so $w_j>0$ and $\theta_j\in[0,1]$ is well defined. At the initial state, $B_0=F_0=1$, so $\theta_0=1$; at the sentinel, $G_k<0$ almost surely, so $F_k=A_k=u_k=0$ and $\theta_k=0$.

\begin{lemma}[Local insertion]
    \label{lem:cutlift}
    Let $g:2^{[k]}\to\RR$ be increasing with respect to set inclusion. With $\lambda_J=\theta_J-\theta_{J+1}$, for $0\le J\le k-1$, one has $\lambda_J\ge0$, $\sum_{J=0}^{k-1}\lambda_J=1$, and
    \begin{equation}
        \label{eq:cutliftReward}
        \Ee_{\widehat\nu}\,g
        \le
        \sum_{J=0}^{k-1}\lambda_J\,\Ee_{\nu_{\mathcal C^J}}g,
    \end{equation}
    where $\mathcal C^J$ is the chain \eqref{eq:liftedpath} with $X_k$ inserted after state $C_J$. In particular, at least one insertion position~$J$ satisfies $\Ee_{\nu_{\mathcal C^J}}g\ge \Ee_{\widehat\nu}\,g$.
\end{lemma}

\begin{proof}
    By \eqref{eq:chainmeasure} and $F_j=K_{k-1}(C_j)$, the measure $\nu_{\mathcal C}$ places mass $F_j-F_{j+1}$ at $C_j$ for every $0\le j\le k-1$, and at the last state the mass is $K_{k-1}(C_{k-1})=F_{k-1}-F_k=F_{k-1}$, since $F_k=0$ (the sentinel being negative); the independent reveal of $X_k$ then splits each mass:
    \begin{equation}
        \label{eq:hatnu}
        \widehat\nu(\{C_j\})=(1-p)\,(F_j-F_{j+1}),
        \qquad
        \widehat\nu(\{H_j\})=p\,(F_j-F_{j+1}),
        \qquad 0\le j\le k-1.
    \end{equation}

    Let $\bar\nu=\sum_{J=0}^{k-1}\lambda_J\,\nu_{\mathcal C^J}$, so that $\sum_{J=0}^{k-1}\lambda_J\,\Ee_{\nu_{\mathcal C^J}}g=\Ee_{\bar\nu}\,g$. Along $\mathcal C^J$, the values of $K_k$ are $B_j$ at $C_j$ and $A_j$ at $H_j$, so by \eqref{eq:chainmeasure} its chain measure places mass $B_j-B_{j+1}$ at $C_j$ for $j<J$, mass $B_J-A_J=w_J$ at $C_J$, and mass $A_j-A_{j+1}$ at $H_j$ for $j\ge J$, the mass at the last state $H_{k-1}$ being $A_{k-1}=A_{k-1}-A_k$, since $A_k=0$. Averaging over $J$, with $\sum_{J>j}\lambda_J=\theta_{j+1}$ and $\sum_{J\le j}\lambda_J=1-\theta_{j+1}$,
    \begin{equation}
        \label{eq:barnu}
        \bar\nu(\{C_j\})=\theta_{j+1}\,(B_j-B_{j+1})+\lambda_j\,w_j,
        \qquad
        \bar\nu(\{H_j\})=(1-\theta_{j+1})\,(A_j-A_{j+1}).
    \end{equation}
    Since $B_j=A_j+w_j$ and $F_j=A_j+\theta_jw_j$, the two masses in \eqref{eq:barnu} add up to $F_j-F_{j+1}$. Thus $\bar\nu$ and $\widehat\nu$ give each pair $\{C_j,H_j\}$ the same mass and differ only in how they split it between the two states, and their difference moves mass only within pairs:
    \begin{equation}
        \label{eq:transport}
        \bar\nu-\widehat\nu
        =\sum_{j=0}^{k-1}\eta_j\,(\delta_{H_j}-\delta_{C_j}),
        \qquad
        \eta_j=(1-\theta_{j+1})\,(A_j-A_{j+1})-p\,(F_j-F_{j+1}),
    \end{equation}
    where $\delta_S$ denotes the point mass at $S$, and the transfer coefficient $\eta_j=\bar\nu(\{H_j\})-\widehat\nu(\{H_j\})$ is the excess mass that $\bar\nu$ places on $H_j$. Consequently
    \begin{equation}
        \label{eq:gapsum}
        \sum_{J=0}^{k-1}\lambda_J\,\Ee_{\nu_{\mathcal C^J}}g
        -\Ee_{\widehat\nu}\,g
        =\sum_{j=0}^{k-1}\eta_j\,\bigl[g(H_j)-g(C_j)\bigr],
    \end{equation}
    with every bracket nonnegative, $g$ being increasing and $C_j\subset H_j$. It remains to show that every transfer is upward: $\eta_j\ge0$ for all $j$.

    Fix $1\le j\le k$. The pair $G_{j-1},G_j$ differs in one exponential $E$, with coefficients $a$ and $-b$:
    \begin{equation}
        \label{eq:edge}
        G_{j-1}=Y+aE,
        \qquad
        G_j=Y-bE,
    \end{equation}
    with the common part $Y$ independent of $E$ and of $E_k$: for $j\le k-1$, $E=E_i$ with $i$ the unique element of $C_j\setminus C_{j-1}$, $a=\gamma_i$, and $b=\beta_i$; for the terminal edge $j=k$, $E=E_0$ and $a=b=1$, by \eqref{eq:slack} and \eqref{eq:slackTop}. In every case $a\ge\gamma_k=c$, by the sorting---at $j=k$ this is the constraint $\gamma_k\le1$, its only use. Lemma~\ref{lem:transfer}, applied with $Z_+=G_{j-1}$, $Z_-=G_j$, and $E'=E_k$, then has $F_\pm,A_\pm,B_\pm,\ldots$ equal to the chain quantities indexed $j-1$ and $j$, and the identity \eqref{eq:transfer} becomes
    \begin{equation}
        \label{eq:etaFormula}
        \eta_{j-1}=\frac{a-c}{c+d}\,w_{j-1}\,(\theta_{j-1}-\theta_j).
    \end{equation}
    All three factors on the right are nonnegative: $a\ge c$ by the sorting, $w_{j-1}>0$, and $\theta_{j-1}\ge\theta_j$, which we now prove. For $j\le k-1$, the common part $Y$ is a sum of independent terms of the forms $E_0$, $\gamma_\ell E_\ell$, and $-\beta_\ell E_\ell$, whose densities are log-concave, and convolution preserves log-concavity; hence $\theta_{j-1}\ge\theta_j$ by \eqref{eq:monge}. At the terminal edge, $\theta_{k-1}\ge0=\theta_k$ directly. Therefore $\theta_0\ge\theta_1\ge\cdots\ge\theta_k$ and, by \eqref{eq:etaFormula}, every $\eta_j\ge0$; with $\theta_0=1$ and $\theta_k=0$, the monotonicity also proves the claims on the weights $\lambda_J=\theta_J-\theta_{J+1}$: they are nonnegative, and $\sum_{J=0}^{k-1}\lambda_J=\theta_0-\theta_k=1$. Hence, \eqref{eq:gapsum} proves \eqref{eq:cutliftReward}, and the final assertion follows from the fact that a convex average cannot exceed the maximum of its terms.
\end{proof}

\section{Reduction to the general theorem}
\label{sec:global}

Here we pass from mean-one two-point laws to arbitrary independent nonnegative variables with means at most one. The mean-one case is handled by decomposing each marginal into a mixture of two-point laws; the general case reduces to it by rescaling.

\begin{lemma}[Two-point mixture]
    \label{lem:decomp}
    Every probability law $\mu$ on $[0,\infty)$ with mean $1$ is a mixture of mean-one laws supported on at most two points, one in $[0,1]$ and one in $[1,\infty)$; the degenerate law $\delta_1$ is allowed as a one-point component of the mixture.
\end{lemma}
\begin{proof}
    For a point $z\in[0,\infty)$, let $\delta_z$ denote the Dirac measure at $z$, defined by $\delta_z(B)=\one_{\{z\in B\}}$ for Borel sets $B\subseteq[0,\infty)$. For $0\le x<1<y$, let
    \begin{equation}
        \label{eq:Qdef}
        Q_{x,y}
        =\frac{y-1}{y-x}\,\delta_x
        +\frac{1-x}{y-x}\,\delta_y,
    \end{equation}
    a probability measure with mean one. We represent $\mu$ as a mixture of $\delta_1$ and laws $Q_{x,y}$, in which, informally, the pair $(x,y)$ is drawn as follows: take two independent draws from $\mu$, one conditioned to lie below $1$ and one conditioned to lie above $1$, and bias the pair by its distance $y-x$.

    Let $X\sim\mu$ and define $M=\Ee(1-X)^+=\Ee(X-1)^+$. This equality follows from $\Ee(X-1)=0$, and the two expectations are finite because $0\le(1-X)^+\le1$ and $(X-1)^+\le X$ with $\Ee X=1$. If $M=0$, then $X=1$ almost surely, $\mu=\delta_1$, and we are done; so assume $M>0$. Then $\mu([0,1))>0$ and $\mu((1,\infty))>0$. Define the measure
    \begin{equation}
        \label{eq:themix}
        \widetilde\mu
        =\mu(\{1\})\,\delta_1
        +\frac{1}{M}
        \iint_{x<1<y}(y-x)\,Q_{x,y}\,\mu(dx)\,\mu(dy),
    \end{equation}
    a combination of mean-one laws on at most two points with nonnegative weights; we show below that the total weight is $1$ and that $\widetilde\mu=\mu$.

    Evaluate $\widetilde\mu$ on $[0,t]$ for $t\ge0$. Since $\delta_z([0,t])=\one_{\{z\le t\}}$, the definition \eqref{eq:Qdef} gives
    \begin{equation}
        \label{eq:atommass}
        (y-x)\,Q_{x,y}([0,t])
        =(y-1)\,\one_{\{x\le t\}}
        +(1-x)\,\one_{\{y\le t\}}.
    \end{equation}
    Let $Y$ be an independent copy of $X$. Integrating \eqref{eq:atommass} against $\mu(dx)\,\mu(dy)$ on $\{x<1<y\}$ is taking the expectation of the corresponding function of $(X,Y)$, and each term factorizes by independence:
    \begin{align}
        \Ee\bigl[(Y-1)\,\one_{\{X\le t\}}\,\one_{\{X<1<Y\}}\bigr]
        &=\Ee(Y-1)^+\;\Pp\{X\le t,\ X<1\}
        =M\,\Pp\{X\le t,\ X<1\},
        \label{eq:termone}\\
        \Ee\bigl[(1-X)\,\one_{\{Y\le t\}}\,\one_{\{X<1<Y\}}\bigr]
        &=\Ee(1-X)^+\;\Pp\{Y\le t,\ Y>1\}
        =M\,\Pp\{X\le t,\ X>1\},
        \label{eq:termtwo}
    \end{align}
    where the last equality holds because $Y$ has the law of $X$. Hence
    \begin{equation}
        \label{eq:mixidentity}
        \widetilde\mu([0,t])
        =\mu(\{1\})\,\one_{\{1\le t\}}
        +\Pp\{X\le t,\ X<1\}
        +\Pp\{X\le t,\ X>1\}
        =\Pp\{X\le t\}.
    \end{equation}
    In particular, letting $t\to\infty$ shows that $\widetilde\mu([0,\infty))=1$, so \eqref{eq:themix} is a probability mixture; and by \eqref{eq:mixidentity} its distribution function agrees with that of $\mu$. A law on $[0,\infty)$ is determined by its distribution function, so $\widetilde\mu=\mu$.
\end{proof}

With the two-point mixture in hand, Theorem~\ref{thm:main} follows.

\begin{proof}[Proof of Theorem~\ref{thm:main}]
    First assume $\Ee X_i=1$ for every $i$. By Lemma~\ref{lem:decomp}, the law of each $X_i$ is that of a two-stage draw: draw a latent parameter $\xi_i$, equal to the symbol $\delta_1$ or to a pair $(x_i,y_i)$ with $0\le x_i<1<y_i$, then draw $X_i$ from $\delta_1$ or $Q_{x_i,y_i}$ accordingly; perform these draws independently across $i$, and write $\xi=(\xi_1,\ldots,\xi_n)$. The map $\xi\mapsto\Pp\{K(X)\le\alpha\mid\xi\}$ is measurable: conditional on~$\xi$, the system is an independent mean-one two-point system of Section~\ref{sec:twopoint}, where $\xi_i=\delta_1$ corresponds to the degenerate two-point parameters $\gamma_i=0$, $\beta_i=1$ (the value of $\beta_i$ is immaterial, since $p_i=0$). The probability is the finite sum $\sum_{S}\pi(\{S\})\,\one_{\{K_\xi(S)\le\alpha\}}$, where $\pi(\{S\})$ are the product weights \eqref{eq:productmeasure}, and $K_\xi$ instantiates \eqref{eq:Kcube} for the parameters $\xi$ at level $n$. The weights are continuous in $\xi$; so is each value $K_\xi(S)$, by dominated convergence in \eqref{eq:Kexp}; and $\one_{\{\cdot\,\le\alpha\}}$ is Borel on $\RR$. Conditional on $\xi$, Proposition~\ref{prop:twopoint} applies (its proof deletes the variables deterministic at $1$ and sorts the rest), giving $\Pp\{K(X)\le\alpha\mid\xi\}\le\alpha$. Averaging over $\xi$ proves the mean-one case.

    If $m_i=\Ee X_i\le1$, rescale: set $X_i'=X_i/m_i$ for $m_i>0$ and $X_i'=1$ for $m_i=0$ (in which case $X_i=0$ almost surely). The $X_i'$ are independent, nonnegative, mean one, and $X_i'\ge X_i$ almost surely. Since $K$ is coordinatewise nonincreasing by \eqref{eq:Kexp}, $K(X')\le K(X)$, so $\{K(X)\le\alpha\}\subseteq\{K(X')\le\alpha\}$, and the mean-one case applied to $X'$ completes the proof.
\end{proof}

\end{document}